\documentclass[aop,preprint]{imsart}
\usepackage{amsthm,amsmath,natbib}
\usepackage[psamsfonts]{amssymb}
\usepackage{epsfig}

\startlocaldefs
\numberwithin{equation}{section}
\newtheorem{thm}{Theorem}[section]
\newtheorem{prop}[thm]{Proposition}
\newtheorem{lem}[thm]{Lemma}
\theoremstyle{remark}
\newtheorem{rem}[thm]{Remark}

\newcommand{\N}{{\mathbb N}}
\newcommand{\Q}{{\mathbb Q}}
\newcommand{\R}{{\mathbb R}}
\newcommand{\cB}{{\mathcal B}}
\newcommand{\cC}{{\mathcal C}}

\newcommand{\cE}{{\mathcal E}}
\newcommand{\cF}{{\mathcal F}}
\newcommand{\cH}{{\mathcal H}}
\newcommand{\cK}{{\mathcal K}}
\newcommand{\cL}{{\mathcal L}}
\newcommand{\cM}{{\mathcal M}}
\newcommand{\cN}{{\mathcal N}}
\newcommand{\cP}{{\mathcal P}}
\newcommand{\bF}{{\mathbf F}}
\renewcommand{\a}{\alpha}
\renewcommand{\b}{\beta}
\newcommand{\gm}{\gamma}
\newcommand{\dl}{\delta}
\newcommand{\eps}{\varepsilon}
\newcommand{\zt}{\zeta}

\newcommand{\lm}{\lambda}
\newcommand{\h}{\eta}
\newcommand{\sg}{\sigma}
\newcommand{\ph}{\varphi}
\newcommand{\om}{\omega}
\newcommand{\Sg}{\Sigma}
\newcommand{\Om}{\Omega}
\newcommand{\la}{\langle}
\newcommand{\ra}{\rangle}
\newcommand{\bone}{{\mathbf 1}}
\newcommand{\maruM}{{\stackrel{\circ}{\cM}}}
\endlocaldefs

\begin{document}
\allowdisplaybreaks
\begin{frontmatter}

\title{Martingale dimensions for fractals}
\runtitle{Martingale dimensions for fractals}


\author{\fnms{Masanori} \snm{Hino}\thanksref{t1}\ead[label=e1]{hino@i.kyoto-u.ac.jp}}
\thankstext{t1}{Supported in part by the Ministry of Education, Culture, Sports, Science and Technology, Grant-in-Aid for Encouragement of Young Scientists, 18740070.}
\address{Graduate School of Informatics\\
Kyoto University\\
Kyoto 606-8501, Japan\\
\printead{e1}}

\runauthor{Masanori Hino}

\begin{abstract}
We prove that the martingale dimensions for canonical diffusion processes on a class of self-similar sets including nested fractals are always one.
This provides an affirmative answer to the conjecture of S.~Kusuoka~[Publ.\ Res.\ Inst.\ Math.\ Sci.\ \textbf{25} (1989) 659--680].
\end{abstract}



\end{frontmatter}
\section{Introduction}
The martingale dimension, which is also known as the Davis--Varaiya invariant~\cite{DV} or the multiplicity of filtration, is defined for a filtration on a probability space and represents a certain index for random noises.
Let $(\Omega,\cF,P)$ be a probability space and $\bF=\{\cF_t\}_{t\in[0,\infty)}$ be a filtration on it. 
Informally speaking, the martingale dimension for $\bF$ is the minimal number of martingales $\{M^1,M^2,\ldots\}$ with the property that an arbitrary $\bF$-martingale
$X$ has a stochastic integral representation of the following type: 
\[
X_t=X_0+\sum_{i}\int_0^t \ph^i_s\,dM^i_s.
\]
When $\bF$ is provided by the standard Brownian motion on $\R^d$, its martingale dimension is $d$.
Kusuoka~\cite{Ku1} has proved a remarkable result that when $\bF$ is induced by the canonical diffusion process on the $d$-dimensional Sierpinski gasket, there exists one
martingale additive functional $M$ such that every
martingale additive functional with finite energy is a stochastic integral
of $M$. We will say that the AF-martingale dimension
of $\bF$ is one.
(For remarks about the connection between the
Davis--Varaiya invariant and the
AF-martingale dimension, see comments below Theorem~\ref{thm:main}.)
He has also conjectured that each nested fractal has the same property.
However, with the exception of a few related studies such as \cite{Ku2}, no significant progress has been made thus far with regard to the problem of determining the martingale dimensions for concrete examples of fractals.

In this paper, we solve this problem for a class of fractals including nested fractals by proving that the AF-martingale dimensions are one.
Our method of the proof is different from that of Kusuoka.

This paper is organized as follows.
In Section~2, we provide a framework for the main theorem, and a key proposition is proved in Section~3.
Section~4 describes the analysis of AF-martingale dimensions and the proof of the main theorem. 
In Section~5, we remark on the key proposition.
\section{Framework}
In this section, we provide a framework of Dirichlet forms on self-similar sets according to \cite{Ki}.
Let $K$ be a compact metrizable topological space, $N$ be an integer greater than one, and set $S=\{1,2,\ldots,N\}$.
Further, let $\psi_i\colon K\to K$ be a continuous injective map for $i\in S$.
Set $\Sg=S^\N$. For $i\in S$, we define a shift operator $\sg_i\colon \Sg\to\Sg$ by $\sg_i(\om_1\om_2\cdots)=i\om_1\om_2\cdots$.
Suppose that there exists a continuous surjective map $\pi\colon \Sg\to K$ such that $\psi_i\circ \pi=\pi\circ\sg_i$ for every $i\in S$.
We term $\cL=(K,S,\{\psi_i\}_{i\in S})$ a self-similar structure.

We also define $W_0=\{\emptyset\}$, $W_m=S^m$ for $m\in \N$, and
denote $\bigcup_{m\ge0}W_m$ by $W_*$.
For $w=w_1w_2\cdots w_m\in W_m$, we define $\psi_w=\psi_{w_1}\circ\psi_{w_2}\circ\cdots\circ\psi_{w_m}$, $\sg_w=\sg_{w_1}\circ\sg_{w_2}\circ\cdots\circ\sg_{w_m}$, $K_w=\psi_w(K)$, and $\Sg_w=\sg_w(\Sg)$.
For $w=w_1w_2\cdots w_m\in W_w$ and $w'=w'_1w'_2\cdots w'_{m'}\in W_{w'}$, $ww'$ denotes $w_1w_2\cdots w_mw'_1w'_2\cdots w'_{m'}\in W_{m+m'}$.
For $\om=\om_1\om_2\cdots\in\Sg$ and $m\in\N$, $[\om]_m$ denotes $\om_1\om_2\cdots\om_m\in W_m$.

We set 
\[
\cP=\bigcup_{m=1}^\infty \sg^m\left(\pi^{-1}\left(\bigcup_{i,j\in S,\,i\ne j}(K_i\cap K_j)\right)\right)\quad\text{and}\quad V_0=\pi(\cP),
\]
where $\sg^m\colon\Sigma\to\Sigma$ is a shift operator that is defined by $\sg^m(\om_1\om_2\cdots)=\om_{m+1}\om_{m+2}\cdots$.
The set $\cP$ is referred to as the post-critical set.
In this paper, we assume that $K$ is connected and the self-similar structure $(K,S,\{\psi_i\}_{i\in S})$ is post-critically finite, that is, $\cP$ is a finite set.

A nested fractal (\cite{Li}) is a typical example of post-critically finite self-similar structures.
For convenience, we explain this concept
(see \cite[p.~117]{Ki} for further comments).
Let $\alpha>1$ and $\psi_i$, $i\in S$, be an $\alpha$-similitude in $\R^d$.
That is, $\psi_i(x)=\alpha^{-1}(x-x_i)+x_i$ for some $x_i\in\R^d$.
There exists a unique nonempty compact set $K$ in $\R^d$ such that $K=\bigcup_{i\in S}\psi_i(K)$.
We assume the following open set condition: there exists a nonempty open set $U$ of $\R^d$ such that $\bigcup_{i\in S}\psi_i(U)\subset U$ and $\psi_i(U)\cap\psi_j(U)=\emptyset$ for any distinct $i,j\in S$.
Let $F_0$ be the set of all fixed points of $\psi_i$'s, $i\in S$.
Then, $\#F_0=N$ (see \cite[Corollary~1.9]{Ku2}).
An element $x$ of $F_0$ is termed an essential fixed point if there exist $i,j\in S$ and $y\in F_0$ such that $i\ne j$ and $\psi_i(x)=\psi_j(y)$.
The set of all essential fixed points is denoted by $F$.
We refer to $\psi_w(F)$ for $w\in W_n$ as an $n$-cell.
For $x,y\in \R^d$ with $x\ne y$, let $H_{xy}$ denote the hyperplane in $\R^d$ defined as $H_{xy}=\{z\in\R^d\mid|x-z|=|y-z|\}$.
Let $g_{xy}\colon\R^d\to\R^d$ be the reflection in $H_{xy}$.
We call $K$ a nested fractal if the following conditions hold:
\begin{itemize}
\item $\# F\ge2$;
\item (Connectivity) for any two 1-cells $C$ and $C'$, there exists a sequence of 1-cells $C_i$ $(i=0,\ldots,k)$ such that $C_0=C$, $C_k=C'$ and $C_{i-1}\cap C_i\ne\emptyset$ for all $i=1,\ldots,k$;
\item (Symmetry) for any distinct $x,y\in F$ and $n\ge0$, $g_{xy}$ maps $n$-cells to $n$-cells and any $n$-cell that contains elements on both sides of $H_{xy}$ to itself;
\item (Nesting) for any $n\ge1$ and distinct $w,w'\in W_n$, $\psi_w(K)\cap\psi_{w'}(K)=\psi_w(F)\cap\psi_{w'}(F)$.
\end{itemize}
Then, the triplet $(K,S,\{\psi_i\}_{i\in S})$ is a post-critically finite self-similar structure and $V_0=F$.
Figure~\ref{fig:1} shows some typical examples of nested fractals $K$.
The bottom right part is the three-dimensional Sierpinski gasket that is realized in $\R^3$, and the rest are realized in $\R^2$.
\begin{figure}[ht]
\centerline{\epsfig{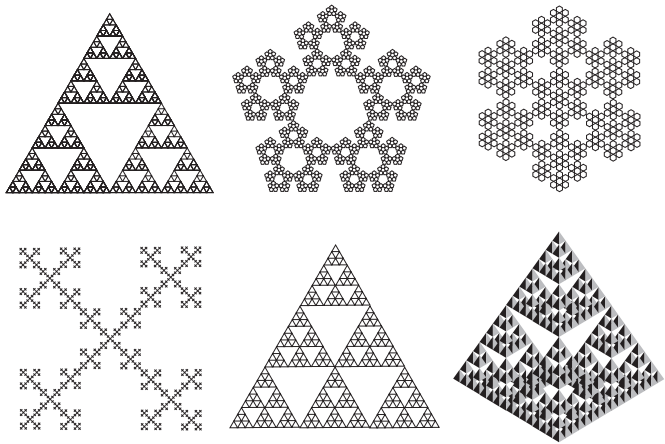}}
\caption{Examples of nested fractals.}
\label{fig:1}
\end{figure}

We resume our discussion for the general case.
For a finite set $V$, let $l(V)$ be the space of all real-valued functions on $V$.
We equip $l(V)$ with an inner product $(\cdot,\cdot)$   defined by $(u,v)=\sum_{p\in V}u(p)v(p)$.
Let $D=(D_{pp'})_{p,p'\in V_0}$ be a symmetric linear operator on $l(V_0)$ (also considered to be a square matrix with size $\#V_0$) such that the following conditions hold:
\begin{enumerate}
\item[(D1)] $D$ is nonpositive definite,
\item[(D2)] $Du=0$ if and only if $u$ is constant on $V_0$,
\item[(D3)] $D_{pp'}\ge0$ for all $p\ne p'\in V_0$.
\end{enumerate}
We define $\cE^{(0)}(u,v)=(-Du,v)$ for $u,v\in l(V_0)$.
This is a Dirichlet form on $l(V_0)$, where $l(V_0)$ is identified with the $L^2$ space on $V_0$ with the counting measure (\cite[Proposition~2.1.3]{Ki}).
Let $V_m=\bigcup_{w\in S^m}\psi_w(V_0)$ for $m\ge1$D
For $r=\{r_i\}_{i\in S}$ with $r_i>0$, we define a bilinear form $\cE^{(m)}$ on $l(V_m)$ as
\begin{equation}\label{eq:em}
  \cE^{(m)}(u,v)=\sum_{w\in W_m}\frac{1}{r_w}\cE^{(0)}(u\circ\psi_w|_{V_0},v\circ\psi_w|_{V_0}),\quad
  u,v\in l(V_m).
\end{equation}
Here, $r_w=r_{w_1}r_{w_2}\cdots r_{w_m}$ for $w=w_1w_2\cdots w_m$.
We refer to $(D,r)$ as a harmonic structure if $\cE^{(0)}(u|_{V_0},u|_{V_0})\le \cE^{(1)}(u,u)$ for every  $u\in l(V_{1})$.
Then, for $m\ge0$ and $u\in l(V_{m+1})$, we obtain $\cE^{(m)}(u|_{V_m},u|_{V_m})\le \cE^{(m+1)}(u,u)$.

We fix a harmonic structure that is regular, namely, $0<r_i<1$ for all $i\in S$.
Several studies have been conducted on the existence of regular harmonic structures. 
We only focus on the fact that all nested fractals have regular harmonic structures (\cite{Ki,Ku2,Li}); we do not go into further details in this regard.

Let $\mu$ be a Borel probability measure on $K$ with full support.
We can then define a regular local Dirichlet form $(\cE,\cF)$ on $L^2(K,\mu)$ as
\begin{align*}
\cF&=\{u\in C(K)\subset L^2(K,\mu)\mid
\lim_{m\to\infty}\cE^{(m)}(u|_{V_m},u|_{V_m})<\infty\},\\
\cE(u,v)&= \lim_{m\to\infty}\cE^{(m)}(u|_{V_m},v|_{V_m}),\quad u,v\in\cF.
\end{align*}
The space $\cF$ becomes a separable Hilbert space when it is equipped with the inner product $\la f,g \ra_\cF=\cE(f,g)+\int_K fg\,d\mu$.
We use $\cE(f)$ instead of $\cE(f,f)$.

For a map $\psi\colon K\to K$ and a function $f\colon K\to\R$, $\psi^* f$ denotes the pullback of $f$ by $\psi$, that is, $\psi^* f=f\circ \psi$.
The Dirichlet form $(\cE,\cF)$ satisfies the following self-similarity:
\begin{align}\label{eq:selfsimilarity}
  \cE(f,g)=\sum_{i\in S}\frac1{r_i}\cE(\psi_i^* f,\psi_i^* g),\quad f,g\in\cF.
 \end{align}
For each $u\in l(V_0)$, there exists a unique function $h\in\cF$ such that $h|_{V_0}=u$ and $h$ attains the infimum of $\{\cE(g)\mid g\in\cF,\ g|_{V_0}=u\}$.
Such a function $h$ is termed a harmonic function.
The space of all harmonic functions is denoted by $\cH$.
By using the linear map $\iota\colon l(V_0)\ni u\mapsto h\in \cH$, we can identify $\cH$ with $l(V_0)$.
In particular, $\cH$ is a finite dimensional subspace of $\cF$.
For each $i\in S$, we define a linear operator $A_i\colon l(V_0)\to l(V_0)$ as $A_i=\iota^{-1}\circ \psi_i^* \circ \iota$.

For $m\ge0$, let $\cH_m$ denote the set of all functions $f$ in $\cF$ such that $\psi_w^*f\in\cH$ for all $w\in W_m$.
Let $\cH_*=\bigcup_{m\ge0}\cH_m$.
The functions in $\cH_*$ are referred to as piecewise harmonic functions.
\begin{lem}
 $\cH_*$ is dense in $\cF$. 
\end{lem}
\begin{proof}
  Let $f\in\cF$. For $m\in\N$, let $f_m$ be a function in $\cH_m$ such that $f_m=f$ on $V_m$.
  Then, $\cE(f-f_m)\to 0$ as $m\to\infty$ by e.g.\ \cite[Lemma~3.2.17]{Ki}.
  From the maximal principle (\cite[Theorem~3.2.5]{Ki}), $f_m$ converges uniformly to $f$. In particular, $f_m\to f$ in $L^2(K,\mu)$ as $m\to\infty$.
  Therefore, $f_m\to f$ in $\cF$ as $m\to\infty$.
\end{proof}

For $f\in\cF$, we will construct a finite measure $\lm_{\la f\ra}$ on $\Sg$ as follows.
For each $m\ge0$, we define
\[
  \lm_{\la f\ra}^{(m)}(A)=2\sum_{w\in A}\frac1{r_w} \cE(\psi_w^* f), \qquad A\subset W_m.
\]
Then, $\lm_{\la f\ra}^{(m)}$ is a measure on $W_m$.
Let $A\subset W_m$ and $A'=\{wi\in W_{m+1}\mid w\in A,\,i\in S\}$. 
Then,
\begin{align*}
  \lm_{\la f\ra}^{(m+1)}(A')
  &= 2\sum_{w\in A}\sum_{i\in S}\frac1{r_{w i}} \cE(\psi_{wi}^* f)\\
  &= 2\sum_{w\in A} \frac1{r_w} \sum_{i\in S}\frac1{r_{i}} \cE(\psi_{i}^*\psi_{w}^* f)\\
  &= 2\sum_{w\in A} \frac1{r_w}  \cE(\psi_{w}^* f)
  \qquad\text{(by \eqref{eq:selfsimilarity})}\\
  &= \lm_{\la f\ra}^{(m)}(A).
\end{align*}
Therefore, $\{\lm_{\la f\ra}^{(m)}\}_{m\ge0}$ has a consistency condition.
We also note that $\lm_{\la f\ra}^{(m)}(W_m)=2\cE(f,f)<\infty$.
  According to the Kolmogorov extension theorem, there exists a unique Borel finite measure $\lm_{\la f\ra}$ on $\Sg$ such that $\lm_{\la f\ra}(\Sg_w)=\lm_{\la f\ra}^{(m)}(\{w\})$ for every $m\ge0$ and $w\in W_m$.
  For $f,g\in\cF$, we define a signed measure $\lm_{\la f,g\ra}$ on $\Sg$ by the polarization procedure; it is expressed as $\lm_{\la f,g\ra}=(\lm_{\la f+g\ra}-\lm_{\la f-g\ra})/4$.
  It is easy to prove that 
\begin{align}\label{eq:consistency}
\lm_{\la f,g\ra}(\Sg_{ww'})=r_w^{-1}\lm_{\la \psi_w^* f,\psi_w^* g\ra}(\Sg_{w'})
\end{align}
for any $f,g\in\cF$ and $w,w'\in W_*$.

  For $f\in\cF$, let $\mu_{\la f\ra}$ be the energy measure of $f$ on $K$ associated with the Dirichlet form $(\cE,\cF)$ on $L^2(K,\mu)$.
  That is, $\mu_{\la f\ra}$ is a unique Borel measure on $K$ satisfying
  \[
    \int_K  \ph\, d\mu_{\la f\ra}= 2\cE(f,f\ph)- \cE(f^2,\ph),
    \qquad \ph\in\cF\subset C(K).
  \]
We define $\mu_{\la f,g\ra}=(\mu_{\la f+g\ra}-\mu_{\la f-g\ra})/4$ for $f,g\in\cF$.
  In the same manner as the proof of \cite[Theorem~I.7.1.1]{BH},
  for every $f\in \cF$, the image measure of $\mu_{\la f\ra}$ by $f$ is proved to be absolutely continuous with respect to the one-dimensional Lebesgue measure.
  In particular, $\mu_{\la f\ra}$ has no atoms.
  From \cite[Lemma~4.1]{Hi}, the image measure of $\lm_{\la f\ra}$ by $\pi\colon\Sg\to K$ is identical to $\mu_{\la f\ra}$.
  Since $\{x\in K\mid \#\pi^{-1}(x)>1\}$ is a countable set, we obtain the following.
\begin{lem}\label{lem:measure}
For any $f\in\cF$, $\lm_{\la f\ra}(\{\om\in\Sg\mid \#\pi^{-1}(\pi(\om))>1\})=0$.
In particular, $(\Sg,\lm_{\la f\ra})$ is isomorphic to $(K,\mu_{\la f\ra})$ as a measure space by the map $\pi\colon\Sg\to K$.
\end{lem}
Hereafter, we assume the following:
\begin{itemize}
\item[$(*)$] Each $p\in V_0$ is a fixed point of $\psi_i$ for some $i\in S$ and $K\setminus \{p\}$ is connected.
\end{itemize}
This condition is a technical one; at present, we cannot remove this in order to utilize Lemma~\ref{lem:useful} below.
A typical example that does not satisfy $(*)$ is Hata's tree-like set (see \cite[Example~1.2.9]{Ki}).
Every nested fractal satisfies this condition (see e.g.\ \cite[Theorem~1.6.2 and Proposition~1.6.9]{Ki} for the proof).
We may and will assume that $V_0=\{p_1,\ldots,p_d\}$ and each $p_i$ is the fixed point of $\psi_i$, 
$i\in \{1,\ldots,d\}\subset S$.

  Let $i\in \{1,\ldots,d\}$. 
We recollect several facts on the eigenvalues and eigenfunctions of $A_i$ in order to use them later.
See \cite[Appendix~A.1]{Ki} and \cite{HN} for further details.
 Both $A_i$ and $\,^t\!A_i$ have 1 and $r_i$ as simple eigenvalues and the modulus of any other eigenvalue is less than $r_i$.
 Let $u_i$ be the column vector $(D_{pp_i})_{p\in V_0}$.
Then, $u_i$ is an eigenvector of $\,^t\!A_i$ with respect to $r_i$ (\cite[Lemma~5]{HN}).
We can take an eigenvector $v_i$ of $A_i$ with respect to $r_i$ so that all components of $v_i$ are nonnegative and $(u_i,v_i)=1$.
Since $v_i$ is not a constant vector, we have $-{}^t v_i D v_i>0$.

Let $\bone\in l(V_0)$ be a constant function on $V_0$ with value $1$.
Let $\tilde l(V_0)=\{u\in l(V_0)\mid (u,\bone)=0\}$ and $P\colon l(V_0)\to l(V_0)$ be the orthogonal projection on $\tilde l(V_0)$.
The following lemma is used in the next section.
\begin{lem}[{\cite[Lemmas~6 and 7]{HN}}]\label{lem:useful}
Let $i\in\{1,\ldots,d\}$ and $u\in l(V_0)$. Then,
\begin{enumerate}
\item $\lim_{n\to\infty}r_i^{-n}P A_i^{n}u
= (u_i,u)Pv_i$,
\item $\lim_{n\to\infty}r_i^{-n}\lm_{\la \iota(u)\ra}(\Sg_{\underbrace{\scriptstyle i\cdots i}_n})=-2(u_i,u)^2\,^t v_i Dv_i$.
\end{enumerate}
\end{lem}
\section{Properties of measures on the shift space}
Let $I$ be a finite set $\{1,\ldots,N_0\}$ or a countable infinite set $\N$.
Take a sequence $\{e_i\}_{i\in I}$ of piecewise harmonic functions such that $2\cE(e_i)=1$ for all $i\in I$.
A real sequence $\{a_i\}_{i\in I}$ is fixed such that $a_i>0$ for every $i\in I$ and $\sum_{i\in I} a_i=1$. 
We define $\lm=\sum_{i=1}^\infty a_i \lm_{\la e_i\ra}$, which is a probability measure on $\Sg$.
For $i,j\in I$, it is easy to see that $\lm_{\la e_i,e_j\ra}$ is absolutely continuous with respect to $\lm$.
The Radon--Nikodym derivative $d\lm_{\la e_i,e_j\ra}/d\lm$ is denoted by $Z^{i,j}$.
It is evident that $\sum_{i\in I} a_i Z^{i,i}(\om)=1$ $\lm$-a.s.\ $\om$.
We may assume that this identity holds for all $\om$.
For $n\in\N$, let $\cB_n$ be a $\sg$-field on $\Sg$ generated by $\{\Sg_w\mid w\in W_n\}$.
We define a function $Z_n^{i,j}$ on $\Sg$ as $Z_n^{i,j}(\om)=\lm_{\la e_i,e_j\ra}(\Sg_{[\om]_n})/\lm(\Sg_{[\om]_n})$.
Then, $Z_n^{i,j}$ is the conditional expectation of $Z^{i,j}$ given $\cB_n$ with respect to $\lm$.
According to the martingale convergence theorem, $\lm(\Sg')=1$, where 
\[
\Sg'=\{\om\in\Sg\mid \lim_{n\to\infty}Z_n^{i,j}(\om)=Z^{i,j}(\om)\text{ for all $i$, $j\in I$}\}.
\]
We define 
\[
\cK=\left\{f\in\cH\left|\, \int_K f\,d\mu=0,\ 2\cE(f)=1\right\}\right..
\]
$\cK$ is a compact set in $\cF$.
For $f\in\cH$, we set
\[
\gm(f)=\max\{|(u_i,f|_{V_0})|;\,i=1,\ldots,d\}.
\]
Here, $(\cdot,\cdot)$ denotes the inner product on $\ell(V_0)$.
When $f$ is not constant,
\[
D(f|_{V_0})=\left(\begin{array}{c}(u_1,f|_{V_0})\\ \vdots \\ (u_d,f|_{V_0})\end{array}\right)
\]
is not a zero vector, therefore $\gm(f)>0$.
Due to the compactness of $\cK$ and the continuity of $\gm$, $\dl:=\min_{f\in\cK}\gm(f)$ is greater than 0.
For $f\in\cH$, we set 
\[
\h(f)=\min\{i=1,\ldots,d; |(u_i,f|_{V_0})|=\gm(f)\}.
\]
The map $\cH\ni f\mapsto \h(f)\in\{1,\ldots,d\}$ is Borel measurable.
\begin{lem}\label{lem:conditional}
For $k\in\N$, there exists $c_k\in(0,1]$ such that for any $n\ge m\ge1$ and $e\in \cH_m$, $w\in W_n$,
\begin{align}\label{eq:conditional}
\lm_{\la e\ra}(\{\om_1\om_2\cdots\in\Sg_w\mid \om_{n+j}=\h(\psi_w^* e)\text{ for all }j=1,\ldots,k\})
\ge c_k\lm_{\la e\ra}(\Sg_w).
\end{align}
\end{lem}
\begin{rem}
When $\lm_{\la e\ra}$ is a probability measure, \eqref{eq:conditional} for all $w\in W_n$ is equivalent to
\[
 \lm_{\la e\ra}[\{\om_1\om_2\cdots\in\Sg\mid \om_{n+j}=\h(\psi_w^* e)\text{ for all }j=1,\ldots,k\}\mid \cB_n]\ge c_k\quad
 \lm_{\la e\ra}\text{-a.s.},
\]
where $\lm_{\la e\ra}[\,\cdot\mid \cB_n]$ denotes the conditional probability of $\lm_{\la e\ra}$ given $\cB_n$.
\end{rem}
\begin{proof}
Equation \eqref{eq:conditional} is equivalent to
\[
\lm_{\la \psi_w^*e\ra}(\{\om_1\om_2\cdots\in\Sg\mid \om_{j}=\h(\psi_w^* e)\text{ for all }j=1,\ldots,k\})\ge c_k\lm_{\la \psi_w^*e\ra}(\Sg).
\]
Therefore, it is sufficient to prove that for $f\in\cK$,
\[
\lm_{\la f\ra}(\{\om_1\om_2\cdots\in\Sg\mid \om_{j}=\h(f)\text{ for all }j=1,\ldots,k\})\ge c_k.
\]
Let $i\in\{1,\ldots,d\}$.
From Lemma~\ref{lem:useful} (2), for any $f\in\cH$,
\[
\lim_{n\to\infty}r_i^{-n}\lm_{\la f\ra}(\Sg_{\underbrace{\scriptstyle i\cdots i}_n})=-2(u_i,f|_{V_0})^2\,^t v_i Dv_i.
\]
Therefore, if $(u_i,f|_{V_0})\ne0$ for $f\in\cH$, we obtain $\lm_{\la f\ra}(\Sg_{\underbrace{\scriptstyle i\cdots i}_n})>0$ for sufficiently large $n\,(\!\!{}\ge k)$.
In particular, $\lm_{\la f\ra}(\Sg_{\underbrace{\scriptstyle i\cdots i}_k})\ge\lm_{\la f\ra}(\Sg_{\underbrace{\scriptstyle i\cdots i}_n})>0$.
Let $\cK_i=\{f\in\cK\mid |(u_i,f|_{V_0})|\ge\dl\}$.
Since $\cK_i$ is compact and the map $\cK_i\ni f\mapsto \lm_{\la f\ra}(\Sg_{\underbrace{\scriptstyle i\cdots i}_k})\in\R$ is continuous, $\eps_i:=\min_{f\in\cK_i}\lm_{\la f\ra}(\Sg_{\underbrace{\scriptstyle i\cdots i}_k})$ is strictly positive.
We define $c_k=\min\{\eps_i\mid i=1,\ldots,d\}$.
For $f\in\cK$, let $i=\h(f)$. Since $f\in\cK_i$, we have $\lm_{\la f\ra}(\Sg_{\underbrace{\scriptstyle i\cdots i}_k})\ge \eps_i\ge c_k$. 
This completes the proof.
\end{proof}
We fix $k\in\N$, $n\ge m\ge1$, and $e\in\cH_m$.
For $j\ge n$, let 
\[
\Om_j=\left\{\om\in\Sg\left|\, \sg^{jk}(\om)\not\in\Sg_{\underbrace{\scriptstyle i\cdots i}_k}\text{ where }i=\h\left(\psi^*_{[\om]_{jk}}e\right)\right\}\right. .
\]
Let $M>n$. Since $\bigcap_{j=n}^{M-1}\Om_j$ is $\cB_{Mk}$-measurable, an application of Lemma~\ref{lem:conditional} yields
$\lm_{\la e\ra}(\bigcap_{j=n}^{M}\Om_j)\le (1-c_k)\lm_{\la e\ra}(\bigcap_{j=n}^{M-1}\Om_j)$.
By repeating this procedure, we obtain $\lm_{\la e\ra}(\bigcap_{j=n}^{M}\Om_j)\le(1-c_k)^{M-n}\lm_{\la e\ra}(\Om_n)$.
Letting $M\to\infty$, we have $\lm_{\la e\ra}(\bigcap_{j= n}^\infty \Om_j)=0$.
Therefore, when we set 
\[
\Xi(e)=\left\{\om\in\Sg\left| \begin{array}{l}\text{for all $k\in\N$, infinitely often $j$,}\\\sg^{jk}(\om)\in\Sg_{\underbrace{\scriptstyle i\cdots i}_k}\text{ where }i=\h\left(\psi^*_{[\om]_{jk}}e\right)\end{array}\right\}\right.
\]
for $e\in \cH_*$, then $\lm_{\la e\ra}(\Sg\setminus\Xi(e))=0$.

For $\om\in\Sg$, we set 
\[
M(\om)=\min\{i\in I\mid a_i Z^{i,i}(\om)=\max_{j\in I}a_j Z^{j,j}(\om)\}.
\]
Clearly, $Z^{M(\om),M(\om)}(\om)>0$ and the map $\Sg\ni\om\mapsto M(\om)\in I$ is measurable.

For $\a\in I$, we set 
\[
\Sg(\a)=\{\om\in\Sg'\mid M(\om)=\a\}\cap \Xi(e_\a).
\]
Then, $\lm_{\la e_\a\ra}(\{M(\om)=\a\}\setminus \Sg(\a))=0$. 
Since $\lm_{\la e_\a\ra}(d\om)=Z^{\a,\a}(\om)\,\lm(d\om)$ and $Z^{\a,\a}(\om)>0$ on $\{M(\om)=\a\}$, $\lm_{\la e_\a\ra}$ is equivalent to $\lm$ on $\{M(\om)=\a\}$.
Thus, $\lm(\{M(\om)=\a\}\setminus \Sg(\a))=0$. 
Therefore, we obtain the following.
\begin{lem}\label{lem:full}
$\lm(\Sg\setminus\bigcup_{\a\in I} \Sg(\a))=0$.
\end{lem}
\begin{proof}
It is sufficient to notice that 
$
\Sg\setminus\bigcup_{\a\in I} \Sg(\a)= \bigcup_{\a\in I}(\{M(\om)=\a\}\setminus \Sg(\a))$. 
\end{proof}
We fix $\a\in I$ and $\om\in \Sg(\a)$. 
It is noteworthy that $\lm_{\la e_\a \ra}(\Sg_{[\om]_n})>0$ for all $n$.
Indeed, if $\lm_{\la e_\a \ra}(\Sg_{[\om]_n})=0$ for some $n$, then $\lm_{\la e_\a \ra}(\Sg_{[\om]_m})=0$ for all $m\ge n$, which implies that $Z^{\a,\a}(\om)=0$, thereby resulting in a contradiction.
In particular, $\psi^*_{[\om]_n}e_\a$ is not constant for an arbitrary $n$.

Take an increasing sequence $\{n(k)\}\uparrow\infty$ of natural numbers such that $e_\a\in\cH_{n(1)}$, and for every $k$,
\[
  \sg^{n(k)}\om\in\Sg_{\underbrace{\scriptstyle\h(\psi^*_{[\om]_{n(k)}}e_\a)\cdots\h(\psi^*_{[\om]_{n(k)}}e_\a)}_k}.
\]
By noting that $\h(\psi^*_{[\om]_{n(k)}}e_\a)$ belongs to $\{1,\ldots,d\}$, there exists $\b\in\{1,\ldots,d\}$ such that $\{k\in\N\mid \h(\psi^*_{[\om]_{n(k)}}e_\a)=\b\}$ is an infinite set.
Take a subsequence $\{n(k')\}$ of $\{n(k)\}$ such that $\h(\psi^*_{[\om]_{n(k')}}e_\a)=\b$ for all $k'$.
For $f\in\cF$, we set 
\[
\xi(f)=\left\{\begin{array}{cl}
\left(f-\int_K f\,d\mu\right)\left/\sqrt{2\cE(f)}\right.& \text{if $f$ is not constant}\\
0 & \text{if $f$ is constant.}
\end{array}\right. 
\]
For $i\in I$, if $k'$ is sufficiently large so that $e_i\in\cH_{n(k')}$, then $\xi(\psi^*_{[\om]_{n(k')}}e_i)\in\cK\cup\{0\}$.
By using the diagonal argument if necessary, we can take a subsequence $\{n(k(l))\}$ of $\{n(k')\}$ such that $\xi(\psi^*_{[\om]_{n(k(l))}}e_i)$ converges in $\cF$ as $l\to\infty$ for every $i\in I$.
For notational conveniences, we denote $\xi(\psi^*_{[\om]_{n(k(l))}}e_i)$ by $f^i_l$ and its limit by $f^i$, which belongs to $\cK\cup\{0\}$.
Since $f^\a_l\in\cK$ for every $l$, we have $|(u_\b,f^\a_l|_{V_0})|\ge\dl$ for every $l$, hence $|(u_\b,f^\a|_{V_0})|\ge\dl$.

From Lemma~\ref{lem:useful} (1),
\[
\lim_{k\to\infty}r_\b^{-k}P A_\b^{k}u
= (u_\b,u)Pv_\b
\]
for any $u\in l(V_0)$.
In particular, the operator norms of $r_\b^{-k}PA_\b^k$ are bounded in $k$. Therefore, since $f_l^i|_{V_0}\to f^i|_{V_0}$ as $l\to\infty$, we obtain
\begin{align}
\label{eq:eq1}
\lim_{l\to\infty}r_\b^{-k(l)}P A_\b^{k(l)}f^i_l|_{V_0}
= (u_\b,f^i|_{V_0})Pv_\b,
\end{align}
which implies that
\begin{align}
\label{eq:eq2}
\lim_{l\to\infty}r_\b^{-2k(l)}\cE( (\psi_\b^*)^{k(l)}f^i_l)
&=\lim_{l\to\infty}-r_\b^{-2k(l)}\,{}^t\!(P A_\b^{k(l)} f_l^i|_{V_0})D(P A_\b^{k(l)} f_l^i|_{V_0})\\
&= -(u_\b,f^i|_{V_0})^2\,^t v_\b PDPv_\b\nonumber\\
&= -(u_\b,f^i|_{V_0})^2\,^t v_\b Dv_\b.\nonumber
\end{align}
It should be noted that the right-hand side of \eqref{eq:eq2} does not vanish when $i=\a$, since $|(u_\b,f^\a|_{V_0})|\ge\dl$.

For $i\in I$ and $n\in\N$, define
$y^i_n=\lm_{\la e_i\ra}(\Sg_{[\om]_n})/\lm_{\la e_\a\ra}(\Sg_{[\om]_n})$.
Since 
\[
y^i_n=\frac{\lm_{\la e_i\ra}(\Sg_{[\om]_n})}{\lm(\Sg_{[\om]_n})}\left/\frac{\lm_{\la e_\a\ra}(\Sg_{[\om]_n})}{\lm(\Sg_{[\om]_n})}\right.,
\]
$y^i_n$ converges to $Z^{i,i}(\om)/Z^{\a,\a}(\om)\in[0,\infty)$ as $n$ tends to $\infty$.
We denote $y^i=Z^{i,i}(\om)/Z^{\a,\a}(\om)$.
Clearly, $y^\a=1$.

Suppose $y^i=0$. Then, for any $j\in I$,
\begin{align*}
\left|\frac{Z^{i,j}(\om)}{Z^{\a,\a}(\om)}\right|
&=\lim_{n\to\infty}\left|\frac{\lm_{\la e_i,e_j\ra}(\Sg_{[\om]_n})}{\lm_{\la e_\a\ra}(\Sg_{[\om]_n})}\right|\\
&\le\limsup_{n\to\infty}\left(\frac{\lm_{\la e_i\ra}(\Sg_{[\om]_n})}{\lm_{\la e_\a\ra}(\Sg_{[\om]_n})}\right)^{1/2}\left(\frac{\lm_{\la e_j\ra}(\Sg_{[\om]_n})}{\lm_{\la e_\a\ra}(\Sg_{[\om]_n})}\right)^{1/2}\\
&=\sqrt{y^i y^j}
=0.
\end{align*}
Thus, $Z^{i,j}(\om)=0$.
We set $\tau_i=1$ for later use.

Next, suppose $y^i>0$. 
Note that $\lm_{\la e_i\ra}(\Sg_{[\om]_n})>0$ for any $n$. 
(Indeed, if $\lm_{\la e_i\ra}(\Sg_{[\om]_n})=0$ for some $n$, then $Z^{i,i}(\om)=0$, which implies that $y^i=0$.)
In particular, $\psi_{[\om]_n}^* e_i$ is not a constant function for an arbitrary $n$.
Take a sufficiently large $l_0$ such that $e_i\in\cH_{n(k(l_0))}$ and $y^i_n>0$ for all $n\ge n(k(l_0))$.
For $m\ge l_0$, we define 
\[
x^i_m=y^i_{n(k(m))+k(m)}\left/y^i_{n(k(m))}\right..
\]
Since $\log y^i_n$ converges as $n\to\infty$, $\log x^i_m$ converges to $0$ as $m\to\infty$.
In other words, $\lim_{m\to\infty}x^i_m=1$.
On the other hand, we have
\begin{align*}
x^i_m
&=\left.\frac{\lm_{\la e_i\ra}(\Sg_{[\om]_{n(k(m))+k(m)}})}{\lm_{\la e_\a\ra}(\Sg_{[\om]_{n(k(m))+k(m)}})}\right/\frac{\lm_{\la e_i\ra}(\Sg_{[\om]_{n(k(m))}})}{\lm_{\la e_\a\ra}(\Sg_{[\om]_{n(k(m))}})}\\
&=\left.\frac{\lm_{\la e_i\ra}(\Sg_{[\om]_{n(k(m))+k(m)}})}{\lm_{\la e_i\ra}(\Sg_{[\om]_{n(k(m))}})}\right/\frac{\lm_{\la e_\a\ra}(\Sg_{[\om]_{n(k(m)+k(m))}})}{\lm_{\la e_\a\ra}(\Sg_{[\om]_{n(k(m))}})}\\
&=\left.\frac{r_{[\om]_{n(k(m))}}^{-1}\lm_{\la \psi^*_{[\om]_{n(k(m))}}e_i\ra}(\Sg_{\underbrace{\scriptstyle\b\cdots\b}_{k(m)}})}{r_{[\om]_{n(k(m))}}^{-1}\lm_{\la \psi^*_{[\om]_{n(k(m))}}e_i\ra}(\Sg)}\right/
\frac{r_{[\om]_{n(k(m))}}^{-1}\lm_{\la \psi^*_{[\om]_{n(k(m))}}e_\a\ra}(\Sg_{\underbrace{\scriptstyle\b\cdots\b}_{k(m)}})}{r_{[\om]_{n(k(m))}}^{-1}\lm_{\la \psi^*_{[\om]_{n(k(m))}}e_\a\ra}(\Sg)}\\*
&~\hspace{24em}\text{(from \eqref{eq:consistency})}\\
&=\left.\frac{2r_\b^{-k(m)}\cE( (\psi^*_\b)^{k(m)}\psi^*_{[\om]_{n(k(m))}}e_i)}{2\cE(\psi^*_{[\om]_{n(k(m))}}e_i)}\right/\frac{2r_\b^{-k(m)}\cE( (\psi^*_\b)^{k(m)}\psi^*_{[\om]_{n(k(m))}}e_\a)}{2\cE(\psi^*_{[\om]_{n(k(m))}}e_\a)}\\
&=\frac{\cE((\psi_\b^*)^{k(m)}f^i_m)}{\cE((\psi_\b^*)^{k(m)}f^\a_m)}\\
&\overset{m\to\infty}\longrightarrow 
\frac{-(u_\b,f^i|_{V_0})^2\,^tv_\b Dv_\b}{-(u_\b,f^\a|_{V_0})^2\,^tv_\b Dv_\b}
\qquad \text{(from \eqref{eq:eq2})}\\
&=\frac{(u_\b,f^i|_{V_0})^2}{(u_\b,f^\a|_{V_0})^2}.
\end{align*}
Therefore, $(u_\b,f^i|_{V_0})=\tau_i(u_\b,f^\a|_{V_0})$ where $\tau_i=1$ or $-1$.
Then, from \eqref{eq:eq1},
\begin{equation}\label{eq:eq3}
  \lim_{l\to\infty}r_\b^{-k(l)}P A_\b^{k(l)} f_l^i|_{V_0}
  = \tau_i(u_\b,f^\a|_{V_0})Pv_\b.
\end{equation}

Now, let $i,j\in I$ and suppose $y^i>0$ and $y^j>0$. 
Then,
\[
\lim_{m\to\infty}\frac{\lm_{\la e_i,e_j\ra}(\Sg_{[\om]_{n(k(m))+k(m)}})}{\lm_{\la e_\a\ra}(\Sg_{[\om]_{n(k(m))+k(m)}})}=
\frac{Z^{i,j}(\om)}{Z^{\a,\a}(\om)}.
\]
On the other hand, for sufficiently large $m$, we have
\begin{align*}
\lefteqn{\frac{\lm_{\la e_i,e_j\ra}(\Sg_{[\om]_{n(k(m))+k(m)}})}{\lm_{\la e_\a\ra}(\Sg_{[\om]_{n(k(m))+k(m)}})}}\\
&=\frac{r_{[\om]_{n(k(m))}}^{-1}\lm_{\la \psi^*_{[\om]_{n(k(m))}}e_i,\psi^*_{[\om]_{n(k(m))}}e_j \ra}(\Sg_{\underbrace{\scriptstyle\b\cdots\b}_{k(m)}})}%
{r_{[\om]_{n(k(m))}}^{-1}\lm_{\la \psi^*_{[\om]_{n(k(m))}}e_\a \ra}(\Sg_{\underbrace{\scriptstyle\b\cdots\b}_{k(m)}})}
\qquad \text{(from \eqref{eq:consistency})}\\
&=\frac{\sqrt{2\cE(\psi^*_{[\om]_{n(k(m))}}e_i)}\sqrt{2\cE(\psi^*_{[\om]_{n(k(m))}}e_j)}\lm_{\la f^i_m,f^j_m \ra}(\Sg_{\underbrace{\scriptstyle\b\cdots\b}_{k(m)}})}%
{2\cE(\psi^*_{[\om]_{n(k(m))}}e_\a)\lm_{\left\la f^\a_m \right\ra}(\Sg_{\underbrace{\scriptstyle\b\cdots\b}_{k(m)}})}\\
&=\frac{\sqrt{\lm_{\la e_i\ra}(\Sg_{[\om]_{n(k(m))}})\lm_{\la e_j\ra}(\Sg_{[\om]_{n(k(m))}})}}{\lm_{\la e_\a\ra}(\Sg_{[\om]_{n(k(m))}})}
\cdot\frac{2r_\b^{-k(m)}\cE((\psi^*_\b)^{k(m)}f^i_m,(\psi^*_\b)^{k(m)}f^j_m)}%
{2r_\b^{-k(m)}\cE((\psi^*_\b)^{k(m)}f^\a_m)}\\
&=\sqrt{y^i_{n(k(m))}y^j_{n(k(m))}}\cdot
\frac{-\strut^t\!\left(r_\b^{-k(m)}P A_\b^{k(m)}f^i_m|_{V_0}\right)D\left(r_\b^{-k(m)}P A_\b^{k(m)}f^j_m|_{V_0}\right)}{-\strut^t\!\left(r_\b^{-k(m)}P A_\b^{k(m)}f^\a_m|_{V_0}\right)D\left(r_\b^{-k(m)}P A_\b^{k(m)}f^\a_m|_{V_0}\right)}\\
&\overset{m\to\infty}\longrightarrow 
\sqrt{y^i y^j}\cdot\frac{\tau_i (u_\b,f^\a|_{V_0})\cdot\tau_j (u_\b,f^\a|_{V_0})\,^tv_\b D v_\b}{(u_\b,f^\a|_{V_0})^2\,^tv_\b D v_\b}
\qquad\text{(from \eqref{eq:eq3})}\\
&=\sqrt{y^i y^j}\tau_i\tau_j.
\end{align*}
Therefore, $\sqrt{y^i y^j}\tau_i\tau_j=Z^{i,j}(\om)/Z^{\a,\a}(\om)$.
This relation is valid even when $y^i=0$ or $y^j=0$.

For $i\in I$, we define 
\begin{align}\label{eq:zt}
\zt_i=\frac{Z^{i,\a}(\om)}{\sqrt{Z^{\a,\a}(\om)}}.
\end{align}
 Then,
\begin{align*}
\zt_i\zt_j
&=\frac{Z^{i,\a}(\om)}{Z^{\a,\a}(\om)}\cdot\frac{Z^{j,\a}(\om)}{Z^{\a,\a}(\om)}\cdot Z^{\a,\a}(\om)\\
&=\sqrt{y^i y^\a}\tau_i\tau_\a\cdot\sqrt{y^j y^\a}\tau_j\tau_\a\cdot Z^{\a,\a}(\om)\\
&=\sqrt{y^i y^j}\tau_i\tau_j\cdot Z^{\a,\a}(\om)\\
&=Z^{i,j}(\om).
\end{align*}
Along with Lemma~\ref{lem:full}, we have proved the following key proposition.
\begin{prop}\label{prop:key}
There exist measurable functions $\{\zt_i\}_{i\in I}$ on $\Sg$ such that for every $i,j\in I$, $Z^{i,j}(\om)=\zt_i(\om)\zt_j(\om)$ $\lm$-a.s.\ $\om$.
\end{prop}
\begin{rem}
According to this proposition, when $I$ is a finite set, the matrix $(Z^{i,j}(\om))_{i,j\in I}$ has a rank one $\lm$-a.s.\ $\om$.
In particular, the proposition implies that the matrix $Z(\om)$ defined in \cite{Ku1,Ku2} has rank one a.s.~$\om$.
\end{rem}
\section{AF-martingale dimension}
We use the same notations as those used in Sections 2 and 3.
We take $I=\N$ and a sequence $\{e_i\}_{i\in I}$ of piecewise harmonic functions so that the linear span of $\{e_i\}_{i\in I}$ is dense in $\cF$.
Let $\nu$ denote the induced measure of $\lm$ by $\pi\colon\Sg\to K$.
From Lemma~\ref{lem:measure}, $(K,\nu)$ and $(\Sg,\lm)$ are isomorphic as measure spaces.
For each $i\in\N$, take a Borel function $\rho_i$ on $K$ such that $\zt_i=\rho_i\circ\pi$ $\lm$-a.s., where $\zt_i$ appeared in Proposition~\ref{prop:key}.
For $i,j\in\N$, let $z^{i,j}$ be the Radon--Nikodym derivative $d\mu_{\la e_i,e_j\ra}/d\nu$, which is a function on $K$.
Then, $Z^{i,j}=z^{i,j}\circ\pi$ $\lm$-a.s.\ and the result of Proposition~\ref{prop:key} can be rewritten as
\begin{align}\label{eq:relation}
  z^{i,j}=\rho_i\rho_j\quad \nu\text{-a.s.}
\end{align}
Let $\cL^2(Z)$ be a set of all Borel measurable maps $g=(g_i)_{i\in\N}$ from $K$ to $\R^\N$ such that 
$\int_K \left(\sum_{i=1}^\infty |g_i(x)\rho_i(x)|\right)^2\,\nu(dx)<\infty$.
We define a preinner product $\la\cdot,\cdot\ra_Z$ on $\cL^2(Z)$ by
\[
\la g,g'\ra_{Z}=\frac12\int_K \left(\sum_{i=1}^\infty g_i(x)\rho_i(x)\right)\left(\sum_{i=1}^\infty g'_i(x)\rho_i(x)\right)\nu(dx),
\quad g,g'\in\cL^2(Z).
\]
For  $g,g'\in \cL^2(Z)$, we write $g\sim g'$ if $\sum_{i=1}^\infty (g_i(x)-g'_i(x))\rho_i(x)=0$ $\nu$-a.s.\ $x$.
Then, $\sim$ is an equivalence relation.
We denote $\cL^2(Z)/\!\!\sim$ by $L^2(Z)$, which becomes a Hilbert space with inner product $\la\cdot,\cdot\ra_Z$ by abuse of notation.
We identify a function in $\cL^2(Z)$ with its equivalence class.
It should be noted that $L^2(Z)$ is isomorphic to $L^2(K\to\R,\nu)$ by the map $\{g_i\}_{i\in\N}\mapsto 2^{-1/2}\sum_{i=1}^\infty g_i\rho_i$, since $\sum_{i=1}^\infty |\rho_i(x)|>0$ $\nu$-a.s.\,$x$.

We define
\[
\cC=\left\{g=(g_i)_{i\in\N}\in\cL^2(Z)\left|\begin{array}{l}
\text{$g_i$ is a bounded Borel function on $K$}\\
\text{and there exists some $n\in\N$ such that}\\ \text{$g_i=0$ for all $i\ge n$}
\end{array}\right.
\right\}
\]
and let $\tilde\cC$ be the equivalence class of $\cC$ in $L^2(Z)$.
\begin{lem}\label{lem:dense}
Set $\tilde\cC$ is dense in $L^2(Z)$.
\end{lem}
\begin{proof}
Consider $g=(g_i)_{i\in\N}\in \cL^2(Z)$.
For $m\in\N$, let $g_i^{(m)}(\om)=((-m)\vee g_i(\om))\wedge m$ when $i\le m$ and $g_i^{(m)}(\om)=0$ when $i>m$.
Then, $g^{(m)}=(g_i^{(m)})_{i\in\N}$ belongs to $\cC$ and converges to $g$ in $L^2(Z)$. Therefore, $\tilde\cC$ is dense in $L^2(Z)$.
\end{proof}
  Let us review the theory of additive functionals associated with local and conservative regular Dirichlet forms $(\cE,\cF)$ on $L^2(K,\mu)$ (see \cite[Chapter~5]{FOT} for details).
  From the general theory of regular Dirichlet forms, we can construct a conservative diffusion process $\{X_t\}$ on $K$ defined on a filtered probability space $(\Omega,\cF,P,\{P_x\}_{x\in K},\{\cF_t\}_{t\in[0,\infty)})$ associated with $(\cE,\cF)$.
  Let $E_x$ denote the expectation with respect to $P_x$.
  Under the framework of this paper, the following is a basic fact in the analysis of post-critically finite self-similar sets with regular harmonic structures.
We provide a proof for readers' convenience.
\begin{prop}\label{prop:capacity}
The capacity derived from $(\cE,\cF)$ of any nonempty set in $K$ is uniformly positive.
\end{prop}
\begin{proof}
Since the Dirichlet form $(\cE,\cF)$ is constructed by a regular harmonic structure $(D,r)$, we can utilize \cite[Theorem~3.4]{Ki} to assure that there exists $C>0$ such that
\begin{equation}\label{eq:resistance}
  \left(\max_{x\in K}f(x)-\min_{x\in K}f(x)\right)^2
  \le C\cE(f,f),\quad
  f\in\cF\subset C(K).
\end{equation}
Let $U$ be an arbitrary nonempty open set of $K$.
Let $f$ be a function in $\cF\subset C(K)$ such that $f\ge 1$ $\mu$-a.e.\ on $U$.
If $\min_{x\in K}f(x)\le 1/2$, then from \eqref{eq:resistance}, $\cE(f,f)\ge 1/(4C)$.
Otherwise, since $f>1/2$ on $K$, $\|f\|_{L^2(\mu)}^2\ge 1/4$.
Therefore, the capacity of $U$ is not less than $\min\{1/(4C),1/4\}$. 
This completes the proof.
\end{proof}
From this proposition, the concept of a ``quasi-every point'' is identical to that of ``every point''.
  We may assume that for each $t\in[0,\infty)$, there exists a shift operator $\theta_t\colon \Omega\to\Omega$ that satisfies $X_s\circ\theta_t=X_{s+t}$ for all $s\ge0$.
  A real-valued function $A_t(\om)$, $t\in[0,\infty)$, $\om\in\Omega$, is referred to as an additive functional if the following conditions hold:
  \begin{itemize}
  \item $A_t(\cdot)$ is $\cF_t$-measurable for each $t\ge0$.
  \item There exists a set $\Lambda\in\sigma(\cF_t; t\ge0)$ such that $P_x(\Lambda)=1$ for all $x\in K$, $\theta_t\Lambda\subset\Lambda$ for all $t>0$; moreover, for each $\om\in\Lambda$, $A_\cdot(\om)$ is right continuous and has the left limit on $[0,\infty)$, $A_0(\om)=0$, and
  \[
    A_{t+s}(\om)=A_s(\om)+A_t(\theta_s \om),
    \quad t,s\ge0.
  \]
  \end{itemize}
  A continuous additive functional is defined as an additive functional such that $A.(\om)$ is continuous on $[0,\infty)$ on $\Lambda$.
  A $[0,\infty)$-valued continuous additive functional is referred to as a positive continuous additive functional.
   From \cite[Theorem~5.1.3]{FOT}, for each positive continuous additive functional $A$ of $\{X_t\}$, there exists a measure $\mu_A$ (termed the Revuz measure of $A$) such that the following identity holds for any $t>0$ and nonnegative Borel functions $f$ and $h$ on $K$:
  \begin{align}\label{eq:Revuz}
  \lefteqn{\int_K E_x\left[\int_0^t f(X_s)\,dA_s\right]h(x)\,\mu(dx)}\\
  &=\int_0^t \int_K E_x\left[h(X_s)\right]f(x)\,\mu_A(dx)\,ds.\nonumber
  \end{align}
  Let $P_\mu$ be a probability measure on $\Omega$ defined as $P_\mu(\cdot)=\int_K P_x(\cdot)\,\mu(dx)$.
  Let $E_\mu$ denote the expectation with respect to $P_\mu$.
  We define the energy $e(A)$ of the additive functional $A_t$ as
  \[
  e(A)=\lim_{t\to0}(2t)^{-1}E_\mu(A_t^2)
  \]
  if the limit exists.
  
  Let $\cM$ be the space of martingale additive functionals of $\{X_t\}$ that is defined as
  \[
  \cM=\left\{M\left|\begin{array}{ll}
  \text{$M$ is an additive functional such that for each $t>0$,}\\
  \text{$E_x(M_t^2)<\infty$ and $E_x(M_t)=0$ for every $x\in K$}\end{array}
  \right.\right\}.
  \]
  Due to the (strong) local property of $(\cE,\cF)$, any $M\in\cM$ is a continuous additive functional (\cite[Lemma~5.5.1~(ii)]{FOT}).
  
  Each $M\in\cM$ admits a positive continuous additive functional $\la M\ra$ referred to as the quadratic variation associated with $M$ that satisfies
  \[
  E_x[\la M\ra_t]=E_x[M_t^2], \ t>0\text{ for every $x\in K$},
  \] 
  and the following equation holds 
  \begin{equation}\label{eq:energyidentity}
  e(M)=\frac12 \mu_{\la M\ra}(K). 
  \end{equation}
  We set $\maruM=\{M\in\cM\mid e(M)<\infty\}$.
  Then, $\maruM$ is a Hilbert space with an inner product $e(M,M'):=(e(M+M')-e(M-M'))/4$ (\cite[Theorem~5.2.1]{FOT}).
  
  The space $\cN_c$ of the continuous additive functionals of zero energy is defined as
  \[
  \cN_c=\left\{N\left|
  \begin{array}{l}\text{$N$ is a continuous additive functional},\\
  \text{$e(N)=0$, $E_x[|N_t|]<\infty$ for all $x\in K$ and $t>0$}
  \end{array}\right\}\right..
  \]
  For each $u\in\cF\subset C(K)$, there exists a unique expression as
  \[
  u(X_t)-u(X_0)=M_t^{[u]}+N_t^{[u]},\qquad
  M^{[u]}\in\maruM,~~
  N^{[u]}\in\cN_c.
  \]
  (See \cite[Theorem~5.2.2]{FOT}.)
  From \cite[Theorem~5.2.3]{FOT}, $\mu_{\la M^{[u]}\ra}$ equals $\mu_{\la u\ra}$.
  
  For $M\in\maruM$ and $f\in L^2(K,\mu_{\la M\ra})$, we can define the stochastic integral $f\bullet M$ (\cite[Theorem~5.6.1]{FOT}), which is a unique element in $\maruM$ such that
  \[
  e(f\bullet M,L)=\frac12\int_Kf(x)\mu_{\la M,L\ra}(dx)\quad
  \text{for all }L\in\maruM.
  \]
  Here, $\mu_{\la M,L\ra}=(\mu_{\la M+L\ra}-\mu_{\la M-L\ra})/4$.
  We may write $\int_0^\cdot f(X_t)\,dM_t$ for $f\bullet M$ since $(f\bullet M)_t=\int_0^t f(X_s)\,dM_s$, $t>0$, $P_x$-a.s.\ for all $x\in K$ as long as $f$ is a continuous function on $K$ (\cite[Lemma~5.6.2]{FOT}).
  We follow the standard textbook~\cite{FOT} and use the notation $f\bullet M$ to denote the stochastic integral with respect to martingale additive functionals.
  From \cite[Lemma~5.6.2]{FOT}, we also have
  \begin{align}\label{eq:integral}
  d\mu_{\la f\bullet M,L\ra}=f\cdot d\mu_{\la M,L\ra},\quad
  L\in\maruM.
  \end{align}
  
  Now, for $g=(g_i)_{i\in\N}\in\cC$, we define 
\begin{align}\label{eq:chi}
\chi(g)=\sum_{i=1}^\infty g_i\bullet M^{[e_i]}\in\maruM.
\end{align}
  We note that the sum is in fact a finite sum.
\begin{lem}\label{lem:preserve}
The map $\chi\colon \cC\to\maruM$ preserves the (pre-)inner products.
\end{lem}
\begin{proof}
  Take
 $g=(g_i)_{i\in\N}\in\cC$ and $g'=(g'_i)_{i\in\N}\in\cC$.
  Since $\mu_{\la e_i,e_j\ra}$ is equal to $\mu_{\la M^{[e_i]},M^{[e_j]}\ra}$, we obtain
  \begin{align*}
  \la g,g'\ra_{Z}
  &=\frac12\int_K \left(\sum_{i=1}^\infty g_i(x)\rho_i(x)\right)\left(\sum_{j=1}^\infty g'_j(x)\rho_j(x)\right)\nu(dx)\\
  &=\frac12\sum_{i,j}\int_K g_i(x) g'_j(x) z^{i,j}(x)\,\nu(dx)\quad
  (\text{by \eqref{eq:relation})}\\
  &=\frac12\sum_{i,j}\int_K g_i(x) g'_j(x) \,\mu_{\la e_i,e_j\ra}(dx)\\
  &=\frac12\sum_{i,j}\int_K g_i(x)g_j'(x)\,\mu_{\la M^{[e_i]},M^{[e_j]}\ra}(dx)\\
  &=\sum_{i,j}e(g_i\bullet M^{[e_i]},g'_j\bullet M^{[e_j]})\\
  &=e(\chi(g),\chi(g')).
  \end{align*}
This completes the proof.
\end{proof}
By virtue of \cite[Lemma~5.6.3]{FOT} and the fact that the linear span of $\{e_i\}_{i\in\N}$ is dense in $\cF$, $\chi(\cC)$ is dense in $\maruM$.
Therefore, along with Lemma~\ref{lem:dense} and Lemma~\ref{lem:preserve}, $\chi$ can extend to an isomorphism from $L^2(Z)$ to $\maruM$.
By the routine argument, we can prove \eqref{eq:chi} for all $g\in L^2(Z)$, where the infinite sum is considered in the topology of $\maruM$.

The AF-martingale dimension associated with $(\cE,\cF)$ on $L^2(K,\mu)$ is one in the following sense, which is our main theorem.
\begin{thm}\label{thm:main}
  There exists $M^1\in\maruM$ such that, for any $M\in\maruM$, there exists $f\in L^2(K,\mu_{\la M^1\ra})$ that satisfies $M=f\bullet M^1$.
  Moreover, we can take $M^1$ so that $\mu_{\la M^1\ra}=\nu$.
\end{thm}
This theorem states that every martingale additive functional with finite energy is expressed by a stochastic integral with respect to only one fixed martingale additive functional.
Note that considering martingale additive functionals instead of pure martingales under some $P_x$ seems more natural in the framework of time-homogeneous Markov processes (or of the theory of Dirichlet forms).
Of course, every martingale additive functional is a martingale under $P_x$ for every $x$, but it is doubtful whether a pure martingale (under some $P_x$) is derived from a certain martingale additive functional.
Therefore, the connection between AF-martingale dimensions and the Davis--Varaiya invariants is not straightforward.
A general theory of AF-martingale dimensions has been discussed by Motoo and Watanabe~\cite{MW}, which is prior to the work by Davis and Varaiya~\cite{DV}.
Clarifying the connection between AF-martingale dimensions and the Davis--Varaiya invariants (whose definition seems ``too general'' from our standpoint) should be discussed elsewhere, in a more general framework.
\begin{proof}[Proof of Theorem~\protect\ref{thm:main}]
  First, by taking \eqref{eq:zt} into consideration, we note that
  \[
  0<\sum_{j=1}^\infty a_j \rho_j(x)^2
  =\sum_{j=1}^\infty a_j \frac{z^{j,\a}(x)^2}{z^{\a,\a}(x)}
  \le \sum_{j=1}^\infty a_j z^{j,j}(x)=1.
  \]
  For each $i\in\N$, we define
\[
h_i=a_i\rho_i\left/\left(\sum_{k=1}^\infty a_k \rho_k^2\right)\right..
\]
  Since $h_i\rho_i\ge0$ and $\sum_{i=1}^\infty h_i\rho_i=1$, $h=(h_i)_{i\in\N}$ belongs to $\cL^2(Z)$.
  We also define $M^1\in\maruM$ as the image of the equivalence class of $h$ in $L^2(Z)$ by $\chi$.
  Then, at least formally,
  \begin{align*}
  d\mu_{\la M^1\ra}
  &=d\mu_{\left\la \sum_{i=1}^\infty a_i \rho_i\left(\sum_{k=1}^\infty a_k \rho_k^2\right)^{-1}\bullet M^{[e_i]},
  \sum_{j=1}^\infty a_j \rho_j \left(\sum_{k=1}^\infty a_k \rho_k^2\right)^{-1}\bullet M^{[e_j]}\right\ra}\\
  &=\left(\sum_{k=1}^\infty a_k \rho_k^2\right)^{-2}\sum_{i,j}a_i\rho_i a_j\rho_j \,d\mu_{\la M^{[e_i]},M^{[e_j]}\ra}
  \quad\text{(by \eqref{eq:integral})}\\
  &=\left(\sum_{k=1}^\infty a_k \rho_k^2\right)^{-2}\sum_{i,j}a_i\rho_i a_j\rho_j z^{i,j}\,d\nu\\
  &=\left(\sum_{k=1}^\infty a_k \rho_k^2\right)^{-2}\sum_{i,j}a_i\rho_i^2 a_j\rho_j^2\,d\nu
  \quad\text{(by \eqref{eq:relation})}\\
  &=d\nu.
  \end{align*}
  This calculation is justified by approximating $h$ by the elements in $\cC$ and performing a similar calculation.
  
  Let $g=(g_i)_{i\in\N}\in \cL^2(Z)$ such that $\sum_{j=1}^\infty |g_j\rho_j|$ is a bounded function.
  We define $f=\sum_{j=1}^\infty g_j\rho_j$ and $g'_i=f h_i$ for $i\in\N$.
  We have $\sum_{i=1}^\infty |g'_i\rho_i|=|f|$ and $\sum_{i=1}^\infty g'_i\rho_i=f$, which imply that $g'=(g'_i)_{i\in\N}$ belongs to $\cL^2(Z)$ and $g\sim g'$.
  Then,
  $f\bullet M^1=\sum_{i=1}^\infty f\bullet(h_i\bullet M^{[e_i]})$ and $\chi(g')=\sum_{i=1}^\infty (fh_i)\bullet M^{[e_i]}$.
  According to \cite[Corollary~5.6.1]{FOT}, these two additive functionals coincide.
  In other words, $\chi(g)=f\bullet M^1$.
  We also have 
\[
  \la g,g\ra_{Z}=e(f\bullet M^1)=\int_K f^2\,d\mu_{\la M^1\ra}=\int_K f^2\,d\nu.
\]
Based on the approximation argument using this relation, for general $g\in L^2(Z)$, we can take some $f\in L^2(K,\mu_{\la M^1\ra})$ such that $\chi(g)=f\bullet M^1$.
This completes the proof.
\end{proof}
\begin{rem}
\begin{enumerate}
\item The underlying measure $\mu$ on $K$ does not play an important role in this paper.
\item In ``nondegenerate'' examples of fractals, only a finite number of harmonic functions $\{e_i\}$ are required for the argument in this section. 
Such cases are treated in \cite{Ku1,Ku2}.
However, in order to include ``degenerate'' examples such as  the Vicsek set (example in the bottom left part of  Figure~\ref{fig:1}), it is not sufficient to consider only harmonic functions.
\end{enumerate}
\end{rem}
\section{Concluding remarks}
In this section, we remark on Proposition~\ref{prop:key}.
In Section~3, the functions $\{e_i\}_{i\in I}$ were considered to be piecewise harmonic functions such that $2\cE(e_i)=1$.
In fact, Proposition~\ref{prop:key} is true for an arbitrary choice of $\{e_i\}$ in $\cF$.
More precisely, let $J$ be a finite set $\{1,\ldots,N_0\}$ or an infinite set $\N$. 
Let $\{f_i\}_{i\in J}$ be a sequence in $\cF$.
Take a real sequence $\{b_i\}_{i\in J}$ such that $b_i>0$ for every $i\in J$ and $\hat\lm:=\sum_{i\in J}b_i\lm_{\la f_i\ra}$ is a probability measure on $\Sg$.
For $i,j\in J$, we denote the Radon--Nikodym derivative $\lm_{\la f_i,f_j\ra}/d\hat\lm$ by $\hat Z^{i,j}$ and obtain the following.
\begin{prop}\label{prop:key2}
There exist measurable functions $\{\hat\zt_i\}_{i\in J}$ on $\Sg$ such that, for every $i,j\in J$, $\hat Z^{i,j}(\om)=\hat\zt_i(\om)\hat\zt_j(\om)$ $\hat\lm$-a.s.\ $\om$.
\end{prop}
\begin{proof}
 We may assume that $\int_K f_i\,d\mu=0$ for every $i\in J$ without loss of generality.
 In the setting of Section~3, take $I=\N$ and $\{e_i\}_{i\in I}$ so that $\{e_i\}_{i\in I}$ is dense in $\{f\in\cF\mid \int_K f\,d\mu=0,\,2\cE(f)=1\}$ in the topology of $\cF$. 
 The definitions of $\{a_i\}_{i\in\N}$, $\lm$, $Z^{i,j}$, and $\zeta_i$ are the same as those in Section~3.
 First, we prove the following.
\begin{lem}\label{lem:equivalent}
For any $f\in\cF$, $\lm_{\la f\ra}$ is absolutely continuous with respect to $\lm$.
\end{lem}
\begin{proof}
It should be noted that for any measurable set $A$ of $\Sg$ and $g\in\cF$,
\begin{align}\label{eq:ineq1}
\left|\lm_{\la f\ra}(A)^{1/2}-\lm_{\la g\ra}(A)^{1/2}\right|^2\le \lm_{\la f-g\ra}(A).
\end{align}
Indeed, this is proved from the inequalities $\lm_{\la s f-tg\ra}(A)\ge0$ for all $s,t\in\R$.

For the proof of the claim, we may assume $\int_K f\,d\mu=0$.
Take $c\ge0$ and a sequence of natural numbers $\{n(k)\}_{k=1}^\infty$ such that $g_k:=ce_{n(k)}$ converges to $f$ in $\cF$ as $k\to\infty$.
Suppose $\lm(A)=0$. Then, $\lm_{\la g_{k}\ra}(A)=0$ for all $k\in\N$.
From \eqref{eq:ineq1},
\begin{align*}
\left|\lm_{\la f\ra}(A)^{1/2}\right|^2\le \lm_{\la f-g_{k}\ra}(A)
\le 2\cE(f-g_{k},f-g_{k})\to0
\qquad\text{as }k\to\infty.
\end{align*}
Therefore, $\lm_{\la f\ra}(A)=0$ and we have $\lm_{\la f\ra}\ll\lm$.
\end{proof}
In particular, $\hat\lm\ll\lm$ according to this lemma.
Next, we prove the following.
\begin{lem}\label{lem:ineq2}
For any $f,g\in\cF$,
\begin{align}\label{eq:ineq2}
\left(\sqrt{\frac{d\lm_{\la f\ra}}{d\lm}}-\sqrt{\frac{d\lm_{\la g\ra}}{d\lm}}\right)^2\le \frac{d\lm_{\la f-g\ra}}{d\lm}
\qquad \lm\text{-a.s.}
\end{align}
\end{lem}
\begin{proof}
Since $d\lm_{\la sf-tg\ra}/d\lm\ge0$ $\lm$-a.s.\ for all $s,t\in\R$, we have, for $\lm$-a.s., for all $s,t\in\Q$,
\[
  s^2\frac{d\lm_{\la f\ra}}{d\lm}-2st\frac{d\lm_{\la f,g\ra}}{d\lm}+t^2\frac{d\lm_{\la g\ra}}{d\lm}\ge0.
\]
Therefore, 
\[
\left(\frac{d\lm_{\la f,g\ra}}{d\lm}\right)^2\le \frac{d\lm_{\la f\ra}}{d\lm}\cdot \frac{d\lm_{\la g\ra}}{d\lm} \quad \lm\mbox{-a.s.}
\]
Equation~\eqref{eq:ineq2} is derived from this inequality.
\end{proof}
 For each $i\in J$, take $c_i\ge0$ and a sequence of natural numbers $\{n_i(k)\}_{k=1}^\infty$ such that $c_{i}e_{n_i(k)}$ converges to $f_i$ in $\cF$ as $k\to\infty$.
 Let $g_{i,k}=c_{i}e_{n_i(k)}$.

Let $i,j\in J$ and $\sg\in\{0,\pm1\}$.
From Lemma~\ref{lem:ineq2}, we have
\begin{align*}
\int_\Sg &\left(\sqrt{\frac{d\lm_{\la f_i+\sg f_j\ra}}{d\lm}}-\sqrt{\frac{d\lm_{\la g_{i,k}+\sg g_{j,k}\ra}}{d\lm}}\right)^2\,d\lm\\
&\le \int_\Sg \frac{d\lm_{\la (f_i-g_{i,k})+\sg(f_j-g_{j,k})\ra}}{d\lm}\,d\lm\\
&=2\cE((f_i-g_{i,k})+\sg(f_j-g_{j,k}),(f_i-g_{i,k})+\sg(f_j-g_{j,k}))\\
&\to0 \qquad\text{as }k\to\infty.
\end{align*}
Since $d\lm_{\la g_{i,k}+\sg g_{j,k}\ra}/{d\lm}=(c_{i}\zt_{n_i(k)}+\sg c_{j}\zt_{n_j(k)})^2$ from Proposition~\ref{prop:key}, 
$|c_{i}\zt_{n_i(k)}+\sg c_{j}\zt_{n_j(k)}|$ converges to $\sqrt{d\lm_{\la f_i+\sg f_j\ra}/d\lm}$ in $L^2(\lm)$ as $k\to\infty$.
By the diagonal argument, we may assume that $|c_{i}\zt_{n_i(k)}+\sg c_{j}\zt_{n_j(k)}|$ converges $\lm$-a.s.\ as $k\to\infty$ for all $i,j\in J$ and $\sg\in\{0,\pm1\}$.
In particular, $|c_{i}\zt_{n_i(k)}|$ converges to $\sqrt{d\lm_{\la f_i\ra}/d\lm}$ $\lm$-a.s.
Moreover, for $i,j\in J$, $\lm$-a.s.,
\begin{align}\label{eq:crossterm}
c_{i}\zt_{n_i(k)}c_{j}\zt_{n_j(k)}
&=\frac14\{(c_{i}\zt_{n_i(k)}+c_{j}\zt_{n_j(k)})^2-(c_{i}\zt_{n_i(k)}-c_{j}\zt_{n_j(k)})^2\}\\
&\stackrel{k\to\infty}\longrightarrow\frac14\left(\frac{d\lm_{\la f_i+f_j\ra}}{d\lm}-\frac{d\lm_{\la f_i-f_j\ra}}{d\lm}\right)
= \frac{d\lm_{\la f_i,f_j\ra}}{d\lm}.\nonumber
\end{align}
For $\a\in J$, we define 
\[
\Om(\a)=\left\{\om\in\Sg\left|\, \frac{d\lm_{\la f_i\ra}}{d\lm}(\om)=0\ (i=1,\ldots,\a-1)\text{ and }\frac{d\lm_{\la f_\a\ra}}{d\lm}(\om)>0\right\}\right..
\]
Clearly, $\lm(\{\frac{d\hat\lm}{d\lm}(\om)>0\}\setminus\bigcup_{\a\in J} \Om(\a))=0$.
Let $\a\in J$ and $\om\in\Om(\a)$.
For $k\in\N$, we define
\[
\tau_k(\om)=\left\{
\begin{array}{cl}
1&\text{if }\zt_{n_\a(k)}(\om)\ge0\\
-1&\text{otherwise.}
\end{array}\right.
\]
Then, $\tau_k(\om)c_{\a}\zt_{n_\a(k)}(\om)$ converges to $\sqrt{d\lm_{\la f_\a\ra}/d\lm(\om)}>0$ $\lm$-a.s.\ on $\Om(\a)$.
By combining this with \eqref{eq:crossterm} with $j=\a$, $\tau_k(\om)c_{i}\zt_{n_i(k)}(\om)$ converges $\lm$-a.s.\ on $\Om(\a)$.
We denote the limit by $\tilde\zt_i(\om)$.
Then, from \eqref{eq:crossterm} again, we have ${d\lm_{\la f_i,f_j\ra}}/{d\lm}=\tilde\zt_i\tilde\zt_j$ $\lm$-a.s.\ on $\Om(\a)$, for every $i,j\in J$.
Therefore, by defining
\[
\hat\zt_i(\om)=\left\{\begin{array}{cl}
\tilde\zt_i(\om)\left/\sqrt{\frac{d\hat\lm}{d\lm}(\om)}\right. & \text{if }\frac{d\hat\lm}{d\lm}(\om)>0\\
0&\text{otherwise},
\end{array}\right.
\]
we obtain the claim of the proposition.
\end{proof}

We will turn to the next remark.
Take $f_1,\ldots,f_n\in\cF$ and consider the map $\Phi\colon K\ni x\mapsto (f_1(x),\ldots,f_n(x))\in\R^n$.
Suppose that $\Phi$ is injective. Then, since $\Phi$ is continuous, $K$ and $\Phi(K)$ are homeomorphic.
(For example, when $K$ is a $d$-dimensional Sierpinski gasket, $n=d-1$, and $f_i$ $(i\in\{1,\ldots,d-1\})$ is a harmonic function with $f_i(p_j)=\dl_{ij}$, $j\in\{1,\ldots,d\}$, this is true by \cite[Theorem~3.6]{Ki1}.)
Take $a_i>0$, $i=1,\ldots,n$ such that $\nu:=\sum_{i=1}^n a_i\mu_{\la f_i\ra}$ is a probability measure on $K$.
We denote the Radon--Nikodym derivative $d\mu_{\la f_i,f_j\ra}/d\nu$ by $z^{i,j}$, $i,j=1,\ldots,n$.
Let $z(x)=(z^{i,j}(x)))_{i,j=1}^n$.

Let $G$ be a $C^1$-class function on $\R^n$.
We define $g(x)=G(f_1(x),\ldots,f_n(x))$.
Then, $g\in\cF$, and from the chain rule of energy measures of conservative local Dirichlet forms (\cite[Theorem~3.2.2]{FOT}),
\begin{align*}
d\mu_{\la g\ra}
&=d\mu_{\la G(f_1,\ldots,f_n),G(f_1,\ldots,f_n)\ra}\\
&=\sum_{i,j=1}^n \frac{\partial G}{dx_i}(f_1,\ldots,f_n) \frac{\partial G}{dx_j}(f_1,\ldots,f_n) \,d\mu_{\la f_i,f_j\ra}\\
&=\sum_{i,j=1}^n \frac{\partial G}{dx_i}(f_1,\ldots,f_n) \frac{\partial G}{dx_j}(f_1,\ldots,f_n) z^{i,j}\,d\nu\\
&=((\nabla G)(f_1,\ldots,f_n),z(\nabla G)(f_1,\ldots,f_n))_{\R^n}\,d\nu.
\end{align*}
In particular,
\begin{align}\label{eq:gradient}
\cE(g,g)=\frac12\int_K((\nabla G)(f_1,\ldots,f_n),z(\nabla G)(f_1,\ldots,f_n))_{\R^n}\,d\nu.
\end{align}
If we set $\cE'(G,G)=\cE(g,g)$, $z'=z\circ\Phi^{-1}$, and $\nu'=\nu\circ\Phi^{-1}$,
\eqref{eq:gradient} is rewritten as
\[
\cE'(G,G)=\frac12\int_{\Phi(K)}((\nabla G)(y),z'(y)(\nabla G)(y))_{\R^n}\,\nu'(dy).
\]
Since the rank of $z'$ is one $\nu'$-a.s., $z'$ can be regarded as a ``Riemannian metric'' on $\Phi(K)$ and $\Phi(K)$ is considered to be a one-dimensional ``measure-theoretical Riemannian submanifold'' in $\R^n$.
This observation has been stated in \cite{Ki1} in the case of Sierpinski gaskets.


\begin{thebibliography}{10}
\bibitem{BH} 
\textsc{Bouleau, N.} and \textsc{Hirsch, F.} (1991).
\textit{Dirichlet forms and analysis on Wiener space.}
de Gruyter Studies in Mathematics \textbf{14}, 
Walter de Gruyter, Berlin.
\MR{1133391}

\bibitem{DV}
\textsc{Davis, M. H. A.} and \textsc{Varaiya, P.} (1974).
 The multiplicity of an increasing family of $\sigma $-fields. 
\textit{Ann. Probab.}  
\textbf{2} 958--963.
\MR{0370754}

\bibitem{FOT}
\textsc{Fukushima, M.}, \textsc{Oshima, Y.} and \textsc{Takeda, M.} (1994).
\textit{Dirichlet forms and symmetric Markov processes.}
de Gruyter Studies in Mathematics
\textbf{19},
Walter de Gruyter, Berlin.
\MR{1303354}

\bibitem{Hi}
\textsc{Hino, M.} (2005).
 On singularity of energy measures on self-similar sets.
\textit{Probab. Theory Related Fields}  
\textbf{132} 265--290.
\MR{2199293}

\bibitem{HN}
\textsc{Hino, M.} and \textsc{Nakahara, K.} (2006).
 On singularity of energy measures on self-similar sets II.
\textit{Bull. London Math. Soc.}
\textbf{38} 1019--1032.
\MR{2285256}

\bibitem{Ki1}
\textsc{Kigami, J.} (1993).
 Harmonic metric and Dirichlet form on the Sierpi\'nski gasket, in
\textit{Asymptotic problems in probability theory: stochastic models and  diffusions on fractals (Sanda/Kyoto, 1990)},  201--218,
\textit{Pitman Res. Notes Math. Ser.} \textbf{283}, 
Longman Sci. Tech., Harlow, 1993.  
\MR{1354156}

\bibitem{Ki}
\textsc{Kigami, J.} (2001).
\textit{Analysis on fractals}.
Cambridge University Press, Cambridge.
\MR{1940042}

\bibitem{Ku1}
\textsc{Kusuoka, S.} (1989).
 Dirichlet forms on fractals and products of random matrices.
\textit{Publ. Res. Inst. Math. Sci.}  
\textbf{25} 659--680.
\MR{1025071}

\bibitem{Ku2}
\textsc{Kusuoka, S.} (1993).
 Lecture on diffusion processes on nested fractals, in \textit{Statistical mechanics and fractals}, 39--98, Lecture Notes in Math. \textbf{1567}, Springer-Verlag, Berlin.
\MR{1295841}

\bibitem{Li}
\textsc{Lindstr\o m, T.} (1990).
 Brownian motion on nested fractals. 
\textit{Mem. Amer. Math. Soc.}
\textbf{83}, no. 420.
\MR{0988082}

\bibitem{MW}
\textsc{Motoo, M.} and \textsc{Watanabe, S.}  (1964/1965).
 On a class of additive functionals of Markov processes. 
\textit{J. Math. Kyoto Univ.}
\textbf{4} 429--469.
\MR{0196808}
\end{thebibliography}
\end{document}